\documentclass[12pt,a4paper]{amsart}
\usepackage{amsfonts}
\usepackage{amssymb}
\usepackage{amsthm}
\usepackage[centertags]{amsmath}
\usepackage{newlfont}
\usepackage{graphicx}

\newtheorem{theorem}{Theorem}
\newtheorem{proposition}[theorem]{Proposition}
\newtheorem{corollary}[theorem]{Corollary}

\newtheorem{lemma}[theorem]{Lemma}

\def\r{\mathbb R}

\date{}

\begin{document}
\title{Capillary surfaces of constant mean curvature in a right solid cylinder }
\author{Rafael L\'opez}
 \address{Departamento de Geometr\'{\i}a y Topolog\'{\i}a\\
Universidad de Granada\\
18071 Granada, Spain\\}
 \email{ rcamino@ugr.es}
 \author{Juncheol Pyo}
\address{Department of Mathematics\\ Pusan National University\\ Busan 609-735, Korea}
\email{jcpyo@pusan.ac.kr}

\begin{abstract}In this paper we investigate constant mean curvature surfaces with nonempty boundary in Euclidean space that meet a right cylinder at a constant angle along the boundary. If the surface lies inside of the cylinder, we obtain some results of symmetry by using the Alexandrov reflection method. When the mean curvature is zero, we give sufficient conditions to obtain that the surface is part of a plane or a catenoid.
\end{abstract}


 \maketitle

%%%%%%%%%%%%%%%%%%%%%%%%%%%%%%%%%%%%%%%%%%%%%%%%%
\section{Introduction}
%%%%%%%%%%%%%%%%%%%%%%%%%%%%%%%%%%%%%%%%%%%%%%%%%%%%
In absence of gravity, a given amount of liquid $W$ placed on a solid substrate $\Sigma$ in an equilibrium configuration meets $\Sigma$ in a constant angle $\gamma$. Moreover, the mean curvature $H$ of the air-liquid interface $S$ of the drop is proportional to the pressure change across $S$. In the case that these pressures are constant, $S$ is a surface of constant mean curvature. On the other hand, the shape of a liquid that rises in a capillary tube is modeled by a surface of constant mean curvature and meets the walls of the tube in a constant angle. In general, we refer a capillary surface on $\Sigma$ as a constant mean curvature surface $S$ such that its boundary $\partial S$ lies in $\Sigma$ and the angle between the unit normal of $S$ and $\Sigma$ is constant along $\partial S$.
In this paper we study capillary surfaces in right cylinders. In the Euclidean space $\r^3$ we consider $(x,y,z)$ the usual coordinates, where the $z$-axis indicates the vertical direction. Let $\Pi$ be the plane $z=0$ and $C\subset\Pi$ a simply closed curve. Define the right cylinder on base $C$ as the set $\Sigma=C\times\r$. If $C$ is a circle, we say that $\Sigma$ is a circular right cylinder and denote by $L$ its axis. We call the inside of $\Sigma$ the set $K=\Omega\times\r$, where
$\Omega\subset\Pi$ is the bounded domain by $C$. In the context of capillary surfaces in right cylinders, the case most studied has been when $S$ is a graph on $\Omega$ representing the ascend by capillarity forces of a liquid on the tube $K$ \cite{fi}. Here we consider a more general setting of a parametric surface, that is, an immersion $\phi:S\rightarrow \r^3$ of an oriented surface $S$. If $\phi:S\rightarrow\phi(S)$ is a homeomorphism, we say that $S$ is an embedded surface in $\r^3$ and we identify $S$ with $\phi(S)$. A liquid drop in the cylinder $K$ is then viewed as a domain $W\subset\r^3$ such that its boundary $\partial W$ is written by $\partial W=T\cup S$, where $T\subset \Sigma$ and $S=\partial W\setminus T$ and $\partial W$ is not smooth along $\partial S$. Physically, $T$ is the wetted region of $\Sigma$ by $W$.

Examples of capillary surfaces in a circular right cylinder $\Sigma$ is an appropriate piece of a Delaunay surface whose axis agrees with the one of $\Sigma$. A Delaunay surface is constructed by rolling a conic along a straight line in the plane and taking the trace of the focus \cite{de}. This trace then describes a planar curve which is rotated
about the axis along which it was rolled. These surfaces are the catenoids, unduloids, nodoids, right circular cylinders and spheres. Also, the round disc $\Omega\times\{t\}$ is a minimal surface intersecting orthogonally $\Sigma$. In all these examples, the boundaries of the surfaces are curves homotopic to $C$ on $\Sigma$. Our interest is also addressed on liquid drops where the boundary $\partial S$ is non-homotopic to $C$ on $\Sigma$ as it occurs when we deposit a small volume of liquid in the wall $\Sigma$ of the tube $K$.

In section \ref{sec4}, we study compact capillary surfaces in a right cylinder.  We prove that certain capillary surfaces have some symmetric planes by the Alexandrov reflection method (Theorem \ref{t1}, Theorem \ref{T2}). In case the mean curvature of the capillary surface is zero and the boundary is graph on $\partial\Omega$, we prove that the capillary minimal surface is a vertical translation of $\Omega$ (Theorem \ref{t3}).

In our last section, we consider complete capillary surfaces of zero mean curvature on $\Sigma$. Except $\Pi\setminus\Omega$ meeting $\Sigma$ with a right angle, a first example is obtained when we take a circle $C$ in a catenoid. This circle divides the catenoid into two pieces both are capillary surfaces on $\Sigma=C\times\r$ and only one is included in the outside of $K$.
We prove that these surfaces are the only minimal capillary surfaces lying outside of a right cylinder under reasonable assumptions on the behavior of the end of surfaces (Theorem \ref{t4}).
% \begin{figure}[hbtp]
%\begin{center}
%\includegraphics[ width=.25\textwidth]{fig1.eps}
%\end{center}
%\caption{A drop $\Omega$ supported by a cone $C$. The interface $S$ is a capillary surface}\label{fig1}
%\end{figure}
%\begin{figure}[hbtp]
%\begin{center}
%\includegraphics[width=.8\textwidth]{fig2.eps}
%\end{center}
%\caption{Different configurations of capillary spheres and discs on a circular cone}\label{fig2}
%\end{figure}
%%%%%%%%%%%%%%%%%%%%%%%%
\section{ Compact capillary surfaces in cylinders}\label{sec4}
%%%%%%%%%%%%%%%%%%%%%%%%%%%%%%
In this section we consider a compact capillary embedded surface $S$ whose boundary $\partial S$ lies in a right cylinder $\Sigma$. Denote $\partial S=\Gamma_1\cup\ldots\cup\Gamma_n$ the decomposition of $\partial S$ into its connected components.
In our techniques we shall use the Hopf maximum principle of elliptic equations of divergence type, that in our context of cmc surfaces, it is the so-called {\it tangency principle}.
\begin{proposition}[Tangency principle] Let $S_1$ and $S_2$ be two orientable surfaces of $\r^3$ that are tangent at some common point $p$. Assume that $p$ lies in the interiors of both $S_1$ and $S_2$ or $p\in\partial S_1\cap\partial S_2$ and the tangent lines of $\partial S_1$ and $\partial S_2$ coincide at $p$. Let us orient $S_1$ and $S_2$ such that both orientations $N(p)$ agree. Assume that with respect to the reference system determined by $N(p)$, the surface $S_1$ lies above $S_2$ near $p$. If $H_1\leq H_2$ at $p$, then $S_1$ and $S_2$ coincide in an open around $p$.
\end{proposition}
\begin{theorem}\label{t1} Let $\Sigma$ a right cylinder and let $S$ be a capillary embedded compact surface on $\Sigma$ and included in $K$. Assume that all the curves $\Gamma_i$ are nullhomotopic curves on $\Sigma$. Then $S$ has a plane of symmetry parallel to $\Pi$. In addition, if $\Sigma$ is a circular right cylinder with axis $L$ and $\partial S$ is strictly contained in a halfcylinder determined by a plane containing $L$, then $S$ has a plane of symmetry containing $L$ and $S$ is topologically a disc.
\end{theorem}
\begin{proof} Because the curves $\Gamma_i$ are nullhomotopic on $\Sigma$, the surface $S$ together a bounded set $\Lambda\subset\Sigma$ defines a $3$-domain $W\subset K$. Orient $S$ by the Gauss map $N$ pointing to $W$ and let $H$ be the mean curvature of $S$. Consider the uniparametric family of horizontal planes $\{\Pi(t);t\in\r\}$ where $\Pi(t)$ is the plane of equation $z=t$ and we apply the Alexandrov reflection method \cite{al}. See Figure \ref{fig1}, left. We describe briefly the main ingredients of the technique. If $p\in\r^3$, we denote $p=(p_1,p_2,p_3)$. We introduce the next notation. If $A\subset\r^3$, let $A(t)^{-}=\{p\in A: p_3> t\}$ and $A(t)^{+}=\{p\in A: p_3< t\}$. Let $\hat{A}(t)$ be the reflection of $A(t)^+$ about $\Pi(t)$. We write $A\leq B$ if for every $p\in A$, $q\in B$ such that $(p_1,p_2)=(q_1,q_2)$, we have $p_3\leq q_3$.
Because $\overline{W}$ is compact, consider $t$ close to $t=-\infty$ such that $S(t)^{-}=\emptyset$. We increase the value $t$ until the first value $t_0>0$ such that $\Pi(t_0)$ touches $\overline{W}$. Because $S$ is embedded, for a small $\epsilon>0$, we have $$\hat{S}(t)\leq S(t)^{+}\ \ \mbox{and}\ \ \hat{S}(t)\subset W,$$
for every $t\in (t_0,t_0+\epsilon)$. For these values of $t$, the surface $S(t)^{-}$ is a graph on $\Pi(t)$. Next we move up $\Pi(t)$ and reflecting $S(t)^{-}$ across $\Pi(t)$ until the first touching point $p$ between $\hat{S}(t)$ and $S(t)^{+}$. Let
$$t_1=\sup\{t; \hat{S}(t)\leq S(t)^{+}, \hat{S}(t)\subset W\}.$$
We analyze the different possibilities of the point $p$.
\begin{enumerate}
\item The point $p$ is an interior point of $\hat{S}(t_1)$ and $S(t_1)^+$ with $p\not\in \Pi(t_1)$. The mean curvature of $\hat{S}(t_1)$ is $H$ with the orientation pointing to $W$. Then tangency principle implies that the surfaces $\hat{S}(t_1)$ and $S(t_1)^+$ agree in an open of $p$ and by connectedness, $\hat{S}(t_1)=S(t_1)^+$. This means that $\Pi(t_1)$ is a plane of symmetry of $S$, proving the result.
\item The point $p$ is an interior point of $\hat{S}(t_1)$ and $S(t_1)^+$ with $p\in \Pi(t_1)$. The boundary version of the tangency principle implies $\hat{S}(t_1)=S(t_1)^+$ again.
\item The point $p$ is a boundary point of $S$. The surfaces $\hat{S}(t_1)$ and $S(t_1)^+$ are tangent at $p$ where both surfaces may be expressed locally as a graph over a corner domain in the common tangent plane at $p$. We apply a version of the Hopf boundary point called the Serrin corner lemma \cite{se2}. The tangency principle holds for this case showing that both surfaces agree in an open around $p$. Then the plane $\Pi(t_1)$ is a plane of symmetry of $S$.
\end{enumerate}
\begin{figure}[hbtp]
\begin{center}
\includegraphics[ width=.8\textwidth]{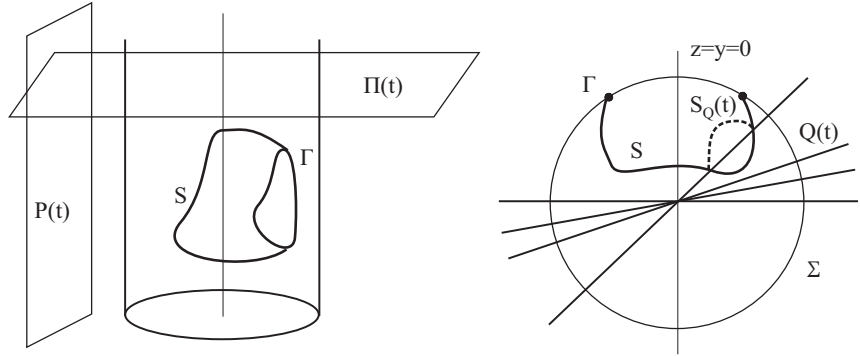}
\end{center}
\caption{Left: the planes $\Pi(t)$ and $P(t)$ in the Alexandrov reflection method; right: the capillary surface $S$ and the planes $Q(t)$ viewed from the top of the cylinder $\Sigma$}\label{fig1}
\end{figure}
For the second part of theorem, we assume that $L$ is the $z$-axis and that $\partial S$ lies in the half cylinder $\Sigma\cap\{x>0\}$. Consider the vertical planes $P(t)$ of equation $x=t$ for $t\in (-\infty,0]$ and the planes $Q(t)$ that contain $L$ parametrized by the angle $t\in [0,\pi/2]$ that makes $Q(t)$ with the vector $(1,0,0)$. See Figure \ref{fig1}. We consider the same notation as above indicating the subscripts $-_{P}$ and $-_{Q}$ is the reflection is done with respect to $P(t)$ or $Q(t)$ respectively. Then we start with the Alexandrov method with reflections across the planes $P(t)$ for value $t$ close to $t=-\infty$. Next, let $t\rightarrow 0$ until we touch $S$ and let follow reflecting the surface behind $P(t)$ about $P(t)$. Assume that there exist $t\leq 0$ such that $\hat{S}_P(t)\leq S_P(t)^{+}$ and
$\hat{S}_P(t)$ and $S_P(t)^{+}$ have a common touching point $p$. Because the cylinder is circular and $\partial S\subset\{x>0\}$, then $p\not\in\partial S$. Then $p$ is an interior point of $S$ and the tangency principle would say that $P(t)$ is a plane of symmetry of $S$, which it is a contradiction with the fact that $\partial S\subset \{x>0\}$. Therefore, we can arrive until $t=0$ having $\hat{S}_P(0)\leq S_P(0)^+$. We remark that it is possible that $\hat{S}_P(0)=\emptyset$, that is, no plane $P(t)$, $t\leq 0$, intersects $S$, which it occurs if $S$ is included in the half cylinder $K\cap\{x>0\}$.
Once arrived at the position $P(0)$, we change the reflection technique by the planes $Q(t)$. We start from $t=0$, reflecting the surface and letting $t\nearrow 0$. We remark that $P(0)=Q(0)$ in such way that the Alexandrov method works. The planes $Q(t)$ sweep all the half cylinder $\Sigma\cap \{x>0\}$. See Figure \ref{fig1}, right. Then in the Alexandrov method, necessarily there is a first time $t_1\in (0,\pi/2]$ such that $\hat{S}_Q(t_1)$ and $S_Q(t_1)^+$ have a common touching point. This point can be an interior point of $S$ or a corner point. Anyway the tangency principle assures that $Q(t_1)$ is a plane of symmetry of $S$. The statement on the topology of $S$ is a consequence that the surface $S$ is invariant by the symmetries across the two orthogonal planes $\Pi(t_1)$ and $Q(t_1)$ and the fact that by the Alexandrov method, each one of the parts that both planes divides $S$ is a graph on these planes.
\end{proof}
Since the reflections about the planes $P(t)$ and $Q(t)$ do not change the third spatial coordinate, we conclude:
\begin{corollary} Let $\Sigma$ a circular right cylinder and let $S$ be an embedded compact surface included in $K$ such that its mean curvature $H$ does not depend on $z$. If $\partial S$ is strictly contained in a halfcylinder determined by a plane containing $L$ and $S$ meets $\Sigma$ with constant angle, then there exists a plane $P$ containing $L$ such that $S$ is invariant by the symmetry across $P$.
\end{corollary}
We point out that under a (vertical) gravitational field, the mean curvature $H$ of a liquid drops resting in a support is a linear function on the $z$-coordinate.
\begin{theorem}\label{T2} Let $\Sigma$ a right cylinder and let $S$ be a capillary embedded compact surface on $\Sigma$ and included in $K$.
\begin{enumerate}
\item If $\partial S$ is a connected curve homotopic to $C$ on $\Sigma$, then $S$ is a graph on $\Omega$.
\item If $\partial S=\Gamma_1\cup\Gamma_2$ and each $\Gamma_{i},$ $(i=1,2)$ is homotopic to $C$ on $\Sigma$, then $S$ has a plane of symmetry parallel to $\Pi$.
\end{enumerate}
\end{theorem}
\begin{proof}
\begin{enumerate}
\item The boundary curve $\partial S$ divides $\Sigma$ in two components. Take $\Sigma^{+}$ the component such that $\Sigma^+\cap \Pi(t)=\emptyset$ for some negative number $t$. Then
$S$ together $\Sigma^{+}$ divides $K$ in two unbounded components and we fix $W$ the one such that $W\cap \Pi(t)=\emptyset$ for some negative number $t$. We apply the Alexandrov method by planes $\Pi(t)$ again as in Theorem \ref{t1} starting with planes $\Pi(t)$ and for value $t$ close to $t=-\infty$. Because $S$ is compact, we have the value $t_0$ as in Theorem \ref{t1}. Next, we increase $t$, reflecting $S(t)^{-}$ about $\Pi(t)$ until the value $t_1$ such that $\hat{S}(t_1)$ touches $S(t_1)^+$. Because $W$ is non compact, it can occurs that $t_1=+\infty$, that is, $\hat{S}(t)\leq S(t)^+$ for any $t$. This means that $S$ is a graph on $\Omega$, proving the result. The other possibility is that $t_1<\infty$. The tangency principle implies that there is a horizontal plane of symmetry of $S$, which is not possible because $\partial S$ is connected and homotopic to $C$.
\item The curves $\Gamma_1$ and $\Gamma_2$ define a bounded set $\Lambda\subset \Sigma$ such that together $S$ defines a bounded $3$-domain $W\subset K$. We apply the reflection method as in the first part of Theorem \ref{t1} with the planes $\Pi(t)$. There are three types of contact points, namely, interior and boundary point, and a corner point. In all these cases, $\Pi(t_1)$ is a plane of symmetry of $S$.
\end{enumerate}
\end{proof}
In case the surface has zero mean curvature,
we have a strong result under the hypothesis that the surface is immersed. Previously we recall the flux formula, which holds for surfaces of constant mean curvature.
\begin{lemma} Let $S$ be a compact surface with boundary $\partial S$ and let $\phi:S\rightarrow\r^3$ be an immersion of constant mean curvature $H$. If $\alpha=\phi_{|\partial S}:\partial S\rightarrow\r^3$, then
\begin{equation}\label{flux}
H\int_{\partial S}\langle\alpha(s)\times\alpha'(s),a\rangle\ ds=-\int_{\partial S}\langle N(s)\times\alpha'(s),a\rangle\ ds,
\end{equation}
where $a\in\r^3$, $\times$ is the vectorial product in $\r^3$ and $N$ is the Gauss map corresponding to $H$.
\end{lemma}
\begin{proof} See \cite[Lemma]{lm0}. The 1-form $\omega_p(v)=\langle (H\phi(p)+N(p))\times (d\phi_p)(v),a\rangle$, $v\in T_p S$ is closed because $H$ is constant. Then we apply the Stokes' theorem.
\end{proof}
\begin{theorem}\label{t3} Let $\Sigma$ be a right cylinder and let $\phi:S\rightarrow\r^3$ be a minimal immersion of a compact surface $S$ such that $\phi(S)$ lies in $K$ and $\phi(S)$ meets $\Sigma$ at a constant angle along $\partial S$. If $\phi(\partial S)$ is a graph on $C$, then $\phi(S)$ describes a (horizontal) planar domain.
\end{theorem}
\begin{proof} Denote $\textbf{n}$ the unit normal vector to $\Sigma$ pointing outside and let $N$ be a Gauss map on $S$. Assume first that $\langle N,\textbf{n}\rangle=0$ along $\partial S$, that is, the contact angle is $\gamma=\pi/2$. Take a horizontal plane $\Pi(t)$ of equation $z=t$ for $t$ sufficiently big such that $\Pi(t)\cap S=\emptyset$. We move $\Pi(t)$ vertically down until intersects with $S$. If this occurs at some interior point and because $\Pi(t)$ and $S$ are both minimal, the tangency principle implies that $S\subset \Pi$, proving the result. On the contrary, we arrive until $\Pi(t)$ touches some boundary point $p\in\partial \Sigma$. In fact, this occurs in the highest point of $S$ with respect to $\Pi$. At such point, the tangent vector to $\partial S$ is horizontal. Since $\textbf{n}$ is also horizontal, then necessarily $N$ is vertical and so (up an orientation on $S$ which does not change the value of the mean curvature) the vector $N$ agrees with the normal to $\Pi$. The tangency principle implies that $S$ is a subset of the plane.
Assume now that $\gamma\not=\pi/2$ and we arrive to a contradiction. We write
$\alpha(s)=\beta(s)+\langle\alpha(s),a\rangle a$, where $\beta$ is the orthogonal projection on $\Pi$ and $a=(0,0,1)$. Then
$\alpha'(s)=\beta'(s)+\langle\alpha'(s),a\rangle a$ and $\beta'(s)\not=0$ since $\partial S$ is a graph on $C$. For each $s\in\r$, we consider $\textbf{t}(s)$ the unit tangent vector to $C$ such that $\textbf{n}\times\textbf{t}=a$. Then $\beta'(s)=\varphi(s)\textbf{t}(s)$ for some non-zero function $\varphi$.
On the other hand, by the flux formula \ref{flux} when $H=0$, we have:
\begin{eqnarray}\label{flux2}
0&=&\int_{\partial S}\langle N(s),\alpha'(s)\times a\rangle\ ds=\int_{\partial S}\langle N(s),\beta'(s)\times a\rangle\ ds\nonumber\\
&=&\int_{\partial S}\varphi(s)\langle N(s),\textbf{n}(s)\rangle\ ds =\cos\gamma\int_{\partial S}\varphi(s)\ ds,
\end{eqnarray}
which it is a contradiction because $\varphi(s)\not=0$ for all $s$.
\end{proof}
We give a result of non-existence of capillary surfaces orthogonally intersecting a right cylinder.
\begin{theorem} Let $\Sigma$ be a right cylinder and let $\phi:S\rightarrow\r^3$ be an immersion of a compact surface $S$ with constant mean curvature $H$, $H\not=0$. Assume that $\phi(S)$ meets $\Sigma$ at a constant angle $\gamma$. If $\phi(\partial S)$ is homotopic to $C$ on $\Sigma$, then $\gamma\not=\pi/2$.
\end{theorem}
\begin{proof} Let $a=(0,0,1)$. Because $\phi(\partial S)$ is homotopic to $C$, the integral $\int_{\partial S}\langle
\alpha(s)\times\alpha'(s),a\rangle\ ds$ in the left hand-side of (\ref{flux}) represents the algebraic area of $C$, that is, up a sign, it is twice the area of $\Omega$. In particular, this integral is not zero. Assume that $\gamma=\pi/2$. We use the notation of the proof of Theorem \ref{t3} writting $\beta'(s)=\varphi(s) \textbf{t}(s)$ where now $\varphi$ is a smooth function on $\partial S$. Then the right hand-side of the flux formula (\ref{flux}) gives
$$\int_{\partial S}\langle N(s)\times\alpha'(s),a\rangle\ ds= \cos(\pi/2)\int_{\partial S}\varphi(s)\ ds=0,$$ obtaining a contradiction.
\end{proof}
%%%%%%%%%%%%%%%%%%%%%%%%
\section{ Complete minimal capillary surfaces outside of cylinders}\label{sec5}
%%%%%%%%%%%%%%%%%%%%%%%%%%%%%%
In this section we consider complete capillary surfaces with zero mean curvature lying the outside of $\Sigma$. This is motivated by the examples as a piece of a catenoid that is the graph on the exterior of a round disc $\Omega\subset\Pi$, which is a minimal surface intersecting the cylinder $\Sigma=\partial\Omega\times\r$ at constant angle $\gamma$ or the planar domain $\Pi\setminus\Omega$ which is a capillary minimal surface in the exterior of $\Sigma$. In the first case, $\gamma\not=\pi/2$ and in the second setting, $\gamma=\pi/2$. More precisely, let $S$ be a punctured surface with single boundary $\partial S$ and $\phi:S\rightarrow \r^3\setminus(\mbox{int}(K))$ be a minimal immersion of surface $S$ such that $\phi(\partial S)$ is homotopic to $C$ and $\phi(S)$ meets $\Sigma$ at a constant angle along $\phi(\partial S)$.
In this case, simply, we say that $S$ is a {\it complete minimal capillary surface situated outside of $\Sigma$}.

The immersion of a small neighborhood of the puncture of $S$, topologically punctured open disk, is the {\it end} of $S$. Geometrically, it is a connected component of $S\setminus N,$ where $N\subset\r^3$ is a sufficient large compact domain.
We say that the end of $S$ is {\it regular at infinity}, if the end is expressed by a graph of the following function
\begin{equation}\label{end}
u(x,y)=\alpha\log\left(\sqrt{x^2+y^2}\right)+\beta+\frac{\gamma_1x+\gamma_2y}{x^2+y^2}+O\left(\frac{1}{x^2+y^2}\right),
\end{equation}
over the exterior of a bounded domain in a plane $\Pi_0$ with constants $\alpha, \beta, \gamma_1$ and $\gamma_2$.
When $\alpha=0$, the end is called a {\it planar end} and if $\alpha\not=0$, we have a {\it catenoidal end} (see \cite{sch}).
We also define that the end of $S$ is {\it parallel to $\Sigma$} if $\Pi_0$ is parallel to $\Pi$.
\begin{lemma}[Maximum principle at infinity \cite{mr}] Let $S_1, S_2\subset \r^3$ be two disjoint, connected, properly immersed minimal surfaces with boundary. If $\partial S_1\not=\emptyset$ or $\partial S_2\not=\emptyset$, then after possibly reindexing, the distance between $S_1$ and $S_2$ is equal to $\inf \{dist(p,q) : p\in \partial S_1, q\in S_2\}$.
\end{lemma}
\begin{theorem}\label{t4}
Let $S$ be a complete embedded minimal capillary surface situated outside of $\Sigma$ and meets $\Sigma$ at a constant angle $\gamma$ along $\partial S$ and parallel to $\Sigma$. Assume that the boundary $\partial S$ is a graph on $C$.
\begin{enumerate}
\item If the end of $S$ is planar then $S$ is part of a parallel plane to $\Pi$.
\item If the end of $S$ is catenoidal then $\gamma\not=\pi/2$.
\item Assume that $\Sigma$ is a circular right cylinder with axis $L$. If the end of $S$ is catenoidal then $S$ is part of a catenoid which has the rotational axis $L$.
\end{enumerate}
\end{theorem}
\begin{proof}
\begin{enumerate}
\item We claim that the contact angle $\gamma$ is $\pi/2$. Because the end of $S$ is parallel to $\Sigma$, we may assume that $\Pi_0=\Pi$ and $x,y$ are coordinates in $\Pi$. Let $C_R\subset\Pi$ be the circle of radius $R$ and centered at the origin of $\Pi$. Since the end of $S$ is planar, for sufficient large $R$, $\Gamma_R=S\cap (C_R\times\r)$ is asymptotic to a circle which is a vertical translation of $C_{R}$ and $S$ asymptotically meets $C_R\times\r$ at a right angle along $\Gamma_R$.
Let $S_R$ be the compact subset of $S$ contained in $C_R\times\r$. For sufficient large $R$, $\partial S_R=\partial S\cup\Gamma_R$. We write, as before, $\alpha(s)=\beta(s)+\langle\alpha(s),a\rangle a$ as a parametrization of $\partial S$ with $a=(0,0,1)$. By the flux formula (\ref{flux}) and because $H=0$, we have
\begin{equation}\label{flux3}
\cos\gamma\int_{\partial S}\varphi(s) ds=\int_{\Gamma_R}\langle \nu,a\rangle ds,
\end{equation}
where $\nu$ is outward unit conormal vector field of $S_{R}$ along $\Gamma_R$.
By (\ref{end}) and direct computation, as the radius $R$ goes to infinity, $\int_{\Gamma_R} \nu ds$ converges to a vector $2\pi\alpha N_0$, where $N_0$ is the unit normal vector of $\Pi_0$.
Since the end of $S$ is planar $(\alpha=0)$ and parallel to $\Sigma$, $\int_{\Gamma_R}\langle \nu,a\rangle ds$ converges to zero as $R\nearrow\infty$. Because $\partial S$ is graph on $C$, $\varphi(s)$ is a non-zero function. Hence, by (\ref{flux3}), $\gamma=\pi/2$.
Once proved that $\gamma=\pi/2$, we now show that $S$ is part of a parallel plane to $\Pi$. Denote $E(t)$ the vertical translation of $\Pi\setminus\Omega$ at the height $z=t$. Since the end of $S$ is regular, the total curvature of $S$ is finite. By \cite[Proposition 11.5]{fang}, $S$ is proper. Because $S$ is a planar end asymptotic to a horizontal plane, for $t$ sufficiently big, $E(t)$ is disjoint from $S$. Next, we descend $E(t)$ by letting $t\searrow 0$ until that $E(t)$ touches $S$ at some time $t=t_0$. Since $S$ and $E(t)$ are proper and is included in the outside of $\Sigma$ with $\partial S,\partial(\Pi\setminus\Omega)\subset\Sigma$, and by the maximum principle at infinity, the contact occurs between boundary points. Since $S$ and $E(t_0)$ meet $\Sigma$ with the constant angle ($\gamma=\pi/2$), the tangency principle implies that both surfaces coincide, that is, $S$ is a vertical displacement of $\Pi\setminus\Omega$.
\item Since the end of $S$ is catenoidal and parallel to $\Sigma$, for sufficient large $R$, $\Gamma_R=S\cap (C_R\times\r)$ is asymptotic to a circle which is a vertical translation of $C_{R}$ and $S$ asymptotically meets $C_R\times\r$ at a constant angle along $\Gamma_R$ different from $\pi/2$. Similarly as above, $\int_{\Gamma_R}\langle \nu,a\rangle ds$ converges to a nonzero constant $2\pi\alpha$ as $R\nearrow\infty$. So, by the equations (\ref{flux2}) and (\ref{flux3}), $\int_{\partial S}\langle N(s)\times\alpha'(s),a\rangle\ ds= \cos\gamma\int_{\partial S}\varphi(s)\ ds$ does not vanish. Hence $\gamma\not=\pi/2$.
\item Denote by $S_\theta$ the rotation of $S$ about $L$ and $S_\theta(t)=S_\theta+(0,0,t)$ the vertical translation of $S_\theta$ of height $t$. Because the right cylinder $\Sigma$ is circular, the boundary of $S_\theta$ remains to be included in $\Sigma$. Similarly, $S$ and $S_\theta(t)$ are proper, because they are minimal surfaces with finite total curvature. Fix $\theta$. We take $t\searrow -\infty$ so $\partial S_\theta$ is disjoint of $\partial S_\theta(t)$, which it is possible because $\partial S$ is compact and the end of $S$ is catenoidal. Next, we move up $S_\theta(t)$ vertically. By the maximum principle at infinity, the first contact point occurs between the boundary points. By tangency principle, both surfaces coincide. Since this property holds for any angle $\theta$, $S$ is a surface of revolution about the axis $L$. Hence $S$ is a part of a catenoid.
\end{enumerate}
\end{proof}

\textbf{Acknowledgement.} The first author is partially supported by MEC-FEDER
grant no. MTM2011-22547 and
Junta de Andaluc\'{\i}a grant no. P09-FQM-5088. Part of this work was realized while the first author was visiting KIAS at Seoul and the Department of Mathematics of the Pusan National University in April of 2012, whose hospitality is gratefully acknowledged.

%%%%%%%%%%%%%%%%%%%%%%%%%%%%%%%%%%%%%%

\end{document}